\documentclass[11pt]{amsart}
\usepackage{amssymb,amsgen,amsbsy,amsopn,amsfonts,graphicx}
\usepackage{amsmath}
\usepackage[dvips]{color}
\usepackage{multicol}
\usepackage{mathrsfs}
\usepackage{nccmath}
\usepackage{enumerate}
\usepackage{hyperref}
\usepackage{amsthm}
\newtheorem{Lemma}{Lemma}[section]

\newtheorem{prop}[Lemma]{Proposition}

\newtheorem{rem}[Lemma]{Remark}

\theoremstyle{definition}

\newtheorem*{MT}{Main Theorem}

\numberwithin{equation}{section}
\numberwithin{Lemma}{section}
\def\R{{\mathbb R}}
\newcommand{\ep}{\epsilon}

\newcommand{\p}{\partial}

\newcommand{\om}{\omega}

\begin{document}

\title[Small Periodic Solution]{Small Amplitude Periodic Solutions of Klein-Gordon Equations}


%

\author[Lu]{Nan Lu}
\address{Department of Mathematics \\
Lehigh University\\ 
14 East Packer Avenue, Bethlehem, PA, 18015}
\email{nal314@lehigh.edu}
\maketitle

\begin{abstract}
We consider a class of nonlinear Klein-Gordon equations
$u_{tt}=u_{xx}-u+f(u)$ and obtain a family of small amplitude periodic solutions, where the temporal and spatial period have different scales. The proof is based on a combination of Lyapunov-Schmidt reduction, averaging and Nash-Moser iteration.
\end{abstract}

\section{Introduction}\label{In}
The nonlinear Klein-Gordon equation
\begin{equation}\label{eq1.1}
u_{tt}=u_{xx}-u+f(u)\ , \ x\in\R,
\end{equation}
is an important model in particle physics, which models the field equation for spineless particles. Classical examples include Sine-Gordon equation and $\phi^4$-model. The main result of this paper is to construct a family of small amplitude periodic (both in time and space) solutions of \eqref{eq1.1}, where the temporal and spatial period have different scales. Moreover, we can approximate such periodic solutions by a simple periodic orbit for a planar system up to exponentially small errors. We will postpone the precise statement until the end of Section \ref{SM} after we introduce some mathematical notations. Throughout this paper, we will assume the nonlinear term $f$ to be analytic and odd in $u$. The analyticity is crucial for us to prove the exponentially small error. The oddness is assumed just for convenience. We will comment on how to deal with general $f$ containing quadratic terms later in this section.

The motivation of this paper originates from the sine-Gordon equation ($f(u)=u-\sin u$)
\begin{equation}\label{eq1.1.2}
u_{tt}=u_{xx}-\sin u,
\end{equation}
which has a family of time periodic solutions (breathers)
\begin{equation}\label{eq1.1.3}
u(x,t)=4\arctan{\frac{\sqrt{1-\om^2}\sin{\om t}}{\om\cosh{\sqrt{1-\om^2}x}}}.
\end{equation}
Clearly, the above formula is only defined for $|\om|<1$. Since \eqref{eq1.1} can be viewed as a perturbation of \eqref{eq1.1.2} for small amplitude solutions, it is natural to ask if \eqref{eq1.1} admits any time periodic solution parameterized by $\om$. The author studied the problem for $\om=\sqrt{1-\ep^2}$ in \cite{Lu13}, where he obtained small amplitude (of order $\ep$) breather solutions with $O(e^{-\frac c\ep})$ tails as $|x|\to\infty$, i.e., the solution is $\frac{2\pi}{\om}$ periodic in time and almost localized in space with exponentially small errors. In this manuscript, we continue our study for $\om=\sqrt{1+\ep^2}$. It turns out the solutions we obtain here have completely different behavior in spatial variable, namely, the solution is also periodic in $x$. 

Since the temporal period is explicitly known, we use the spatial dynamics method (interchanging $x$ and $t$) to rewrite \eqref{eq1.1} as a nonlinear wave equation with periodic boundary condition
\begin{equation}\label{eq1.1.4}
u_{tt}=u_{xx}+u-f(u)\ , \ u(x,t)=u(x+\frac{2\pi}{\om},t),
\end{equation}
where $\om=\sqrt{1+\ep^2}$. By normalizing the spatial period (temporal period for \eqref{eq1.1}) to be $2\pi$, i.e. rescale $x$ to $\om x$, we further transform \eqref{eq1.1.4} to
\begin{equation}\label{eq1.1.5}
u_{tt}=\om^2u_{xx}+u-f(u)\ , \ u(x,t)=u(x+2\pi,t).
\end{equation}
Since the nonlinearity $f$ is odd in $u$, it suffices to restrict $u$ to be odd in $x$. Consequently, the linear operator $\om^2\p_{xx}+1$
has characteristic frequencies $\pm\ep i$ and $\pm\sqrt{\om^2k^2-1}i$ for $k\geq 2$ with multiplicity $1$.


The strategy for finding periodic solutions of \eqref{eq1.1} (\eqref{eq1.1.5} under spatial dynamics formulation) is a combination of singular perturbation theory, averaging, Lyapunov-Schmidt reduction and Nash-Moser iteration. First of all, we observe that the characteristic frequencies of the linear operator $\om^2\p_{xx}+1$ have two scales, namely, one pair of $O(\ep)$-eigenvalues and infinitely many pairs of $O(1)$-eigenvalues. To obtain uniform knowledge in $\ep$, we rescale time in \eqref{eq1.1.5} to blow up small eigenvalues from $O(\ep)$ to $O(1)$, which makes the $O(1)$-eigenvalues become $O(\frac 1\ep)$. With appropriate spatial rescaling, we obtain a singularly perturbed system \eqref{eq2.3} and \eqref{eq2.4}. The singular limit of such system can be rigorously justified as a second order ordinary differential equation \eqref{eq2.12} whose phase plane contains a lot of periodic orbits. Secondly, we perform a sequence of partial normal form transformations to obtain a system whose solutions are exponentially close to the limit equation. Finally, we follow the Lyapunov-Schmidt type argument in \cite{SZ02} and the Nash-Moser iteration in \cite{BB10} to find our periodic solutions near those unperturbed ones. 

The problem of finding periodic solutions to Hamiltonian PDEs has been extensively studied
since the 1960s, see for example \cite{BB03,BB10,H66,K79,R67,R68,R71,R78} and references therein. The first breakthrough on this problem was due to Rabinowitz \cite{R67}. He rephrased the problem as a variational problem and proved the existence of periodic solutions under the monotonicity assumption on the nonlinearity whenever the time period was a rational
multiple of the length of the spatial interval. Subsequently, many authors, such
as Br\' ezis, Coron, Nirenberg etc., have used and developed Rabinowitz’s
variational methods to obtain related results, see \cite{B83,BC81,BN78}. In these papers, the time
period $T$ is required to be a rational multiple of $\pi$. The case in which $T$
is some irrational multiple of $\pi$ has been investigated by
Fe\v{c}kan \cite{F95} and McKenna \cite{M85}. At the end of the 1980s, a different approach
which used the Kolmogorov-Arnold-Moser (KAM) theory was developed from the viewpoint of infinite dimensional dynamical systems by Kuksin \cite{K87} and Wayne \cite{E90}. This method allows one to obtain solutions whose periods are irrational multiples of the length of the spatial interval, and it can also be easily extended to construct quasi-periodic solutions see \cite{P97,GY07,Y07} and references therein. Unlike the variational techniques, the KAM theory only yields solutions of small amplitude. Later, in the original work of Craig-Wayne \cite{CW93}, the existence of periodic solutions for the one-dimensional conservative nonlinear wave equation was also proved by using the Lyapunov-Schmidt method and Newton iterations. For exponential stability of periodic solutions, we refer readers to Bambusi-Nekhoroshev \cite{BN88} and Paleari-Bambusi-Cacciatori \cite{PBC01} and references therein.

The methodology employed in this paper is based on a perturbation argument, which is different from the variational technique and the KAM theory. Even though our solutions still have small amplitudes, which is due to scaling, we actually obtain them from some unperturbed periodic orbits which have large amplitudes. The central idea of KAM theory is to use successive approximate solutions (obtained by normal form transformations) with better accuracy to obtain the exact solution. This method usually requires the analyticity of nonlinearity to assure the convergence of the normal form transformations. Here we assume the nonlinear term to be analytic is to obtain the exponentially small error between solutions of \eqref{eq1.1} and some explicitly given unperturbed periodic orbits. In fact, assuming $f\in C^5$ is enough for the existence part. Although we assume $f$ to odd in $u$, we can still deal with $f$ containing quadratic terms. In that case, one needs to work on even function space, which makes the linear operator $\om^2\p_{xx}+1$ to have an additional characteristic frequencies at $\pm1$. Then one can perform a center manifold reduction to get rid of those frequencies, which reduces the problem to the case in this paper. By using the same method, we can also deal with $\om=\frac{1}{k}\sqrt{1+\ep^2}$ for any $k\geq1$, in which case a finite number of hyperbolic eigenvalues appear. The small divisor problem appears in this context which prevents us to prove the result for all small $\ep$ but rather a subset with almost full measure. We mention that our proof does not rely on the Hamiltonian structure of the problem that much. The only property we use is the existence of some invariant quantity. Some related problems had been discussed by P\" oschel \cite{P97}, Berti-Bolle \cite{BB10} and Berti-Bolle-Procesi \cite{BBP10}. In those cases, the characteristic frequencies are uniformly away from $0$, which makes our problem different from theirs. As we mentioned earlier, the idea of the proof is similar to the one in \cite{SZ02}. Although we have simple geometry on the target space (flat), the technical analysis for finding periodic solutions corresponding to $O(\frac 1\ep)$ frequencies is harder due to the infinite dimensional nature of our problem. We overcome this difficulty by a Nash-Moser interation argument, which is adopted from \cite{BB,BB10} with certain modification. Finally, we point out that our method can be applied to non-Hamiltonian (possibly with non-analytic nonlinearity) systems. Meanwhile, the method gives additional information for the perturbed periodic orbits. 

The rest of the paper is organized as follows. In Section \ref{SM}, we set up our problem in a suitable form and state our main result. In Section \ref{P}, we present the partial normal form transformations. In Section \ref{proof}, we prove the main theorem.

\section{Set Up and Main Result}\label{SM}
In this section, we introduce some notations that will be used throughout this paper. Then we perform scaling transformations to rewrite \eqref{eq1.1.5} in a suitable form. Finally, we state the main results.

The function space we will be working on is the standard Sobolev space consisting of odd periodic functions, namely,
\[H^k\triangleq\{h=\sum_{j=1}^\infty a_j\sin{(j x)}\Big| \sum_{j=1}^{\infty}(1+j^2)^ka_j^2<\infty\}\ ,\ H^0=L^2\]
with norm
\[\|h\|_k\triangleq\|h\|_{H^k}=(\sum_{j=1}^{\infty}(1+j^2)^ka_j^2)^{\frac 12}.\]
Given any $g\in H^k$, we define $Pg$ and $Qg$ as
\[Pg\triangleq\frac{1}{\pi}\int_{-\pi}^{\pi}g(x)\sin x\ dx\ ,\ Q g\triangleq g-(\sin x)Pg,\]
i.e., $P$ is the projection onto $\{\R\sin x\}$ and $Q$ is the orthogonal complement of $P$.

We will use $C,C'$ to denote generic constants which may have different values in different places. However, all of them are independent of the perturbation parameter $\ep$. The norm of an element in a Banach space $X$ will be donated as $\|\cdot\|_X$. A ball centered at the origin with radius in a Banach space $X$ is denoted by $B_r(0,X)$.

We recall that for the nonlinear term $f$, we assume
\begin{enumerate}
\item[(A)] $f$ is odd and holomorphic in $u$. Moreover,
$f'(0)=0,f^{\prime\prime\prime}(0)\ne0$.
\end{enumerate}
Consequently, $Pf$ and $Qf$ are analytic functions of their arguments.
\begin{rem}
The assumption $f^{\prime\prime\prime}(0)\ne0$ is not essential and it can begin with 5th or any other higher order terms. We stick to this case in order to include classical examples like sine-Gordon equation and the $\phi^4$ model.
\end{rem}

As we discussed above, the characteristic frequencies of $\om^2\p_{xx}+1$ have two different scales, namely, $O(\ep)$ and $O(1)$. To separate the small one from others, we write
\begin{equation}\label{eq2.1}
u(x,t)=\tilde v(t)\sin{x}+\tilde w(x,t)\triangleq(Pu)\sin x+Q u(x,t),
\end{equation}
where
$\int_{-\pi}^{\pi} \tilde w(x,t)\sin{x}\ dx=0$. Recall that $\om=\sqrt{1+\ep^2}$, we let
\begin{equation}\label{eq2.2}
\tau=\ep\om t\ ,\
 v(\tau)=\frac{\tilde v(\ep^{-1}\om^{-1}\tau)}{\ep}\ ,\
w(x,\tau)=\frac{\tilde w(x,\ep^{-1}\om^{-1}\tau)}{\ep}.
\end{equation}
The rescaling for temporal variable is to blow up the $O(\ep)$ frequencies to $O(1)$. The rescaling for $u$ is the normalize the amplitude of the solution. Plugging the decomposition and rescaling into \eqref{eq1.1.5}, we obtain
\begin{eqnarray}\label{eq2.3}
 v_{\tau\tau}&=&-\frac{1}{\om^2}v-\frac{1}{\ep^3\om^2}Pf(\ep v\sin x+\ep w),\\\label{eq2.4}
 w_{\tau\tau}&=&\frac{\p_{xx}+\frac{1}{\om^2}}{\ep^2}w-\frac{1}{\ep^3\om^2}Q f(\ep v\sin x+\ep w)\ , \ w(x,\tau)=w(x+2\pi,\tau).
\end{eqnarray}
By the analyticity and leading order term assumption of $f$, it is easy to see $Pf$ and $Qf$ are analytic in $(\ep,v,w)$. The system \eqref{eq2.3} and \eqref{eq2.4} has an invariant quantity (Hamitonian)
\begin{equation}\label{eq2.7}
\begin{aligned}&H(v,v_\tau,w,w_\tau,\ep)\\
=&\frac {v_\tau^2}{2}+\frac{v^2}{2\om^2}+\frac{1}{\ep^4\om^2}PF(\ep v\sin x+\ep w)\\
&+\int_{-\pi}^{\pi}\frac 12 w_\tau^2+\frac{1}{2\ep^2}(w_x^2-\frac{1}{\om^2}w^2)+\frac{1}{\ep^4\om^2}Q F(\ep v\sin x+\ep w)\ dx,
\end{aligned}
\end{equation}
where $F$ is the anti-derivative of $f$ with $F(0)=0$. 

To simplify our notation, we let
\begin{equation}\label{eq2.8}
J_\ep\triangleq \frac {1}{\om^2}+\p_{xx},
\end{equation}
and
\begin{equation}\label{2.8a}
 \tilde f(v,w,\ep)\triangleq \frac{-1}{\ep^2\om^2}Pf(\ep v\sin x+\ep w)\ , \ \tilde g(v,w,\ep)\triangleq \frac{-1}{\ep^2\om^2}Qf(\ep v\sin x+\ep w).
\end{equation}
It is straight forward to verify that for small $\ep$,
\begin{equation}\label{eq2.9}
J_\ep^{-1}\in L(QH^s,QH^{s+2})\ ,\ \|J_\ep^{-1}\|_{L(QH^s,QH^{s+2})}\leq2.
\end{equation}
Using above notations, we can rewrite \eqref{eq2.3} and \eqref{eq2.4} as
\begin{eqnarray}\label{eq2.10}
&& v_{\tau\tau}=-\frac{1}{\om^2}v+\tilde f(u,v,\ep),\\\label{eq2.11}
&& w_{\tau\tau}=\frac{J_\ep}{\ep^2}w+\tilde g(v,w,\ep).
\end{eqnarray}
The system consists of \eqref{eq2.10} and \eqref{eq2.11} is singularly perturbed, where the singular motion corresponds to fast oscillation in the normal direction. The singular limit is formally given by taking the first equation with $w=0$ and $\ep=0$, namely,
\begin{equation}\label{eq2.12}
p_{\tau\tau}=-p-\frac {1}{8} f^{\prime\prime\prime}(0)p^3.
\end{equation}
Here the nonlinear term $\frac {1}{8} f^{\prime\prime\prime}(0)p^3$ is obtained by Taylor's expansion. It turns out this formal limiting system can be rigorously justified by the using the Duhamel's principle and Gronwall's inequality\footnote{Rewrite \eqref{eq2.10}, \eqref{eq2.11} and \eqref{eq2.12} as first order systems}. We note that \eqref{eq2.12} also has an invariant quantity
\begin{equation}\label{eq2.13}
H_\star(p,p_\tau)=\frac 12 p_\tau^2+\frac 12 p^2+\frac{1}{32}f^{\prime\prime\prime}(0)p^4.
\end{equation}
It is clear that if $f^{\prime\prime\prime}(0)>0$, every solution of \eqref{eq2.12} is a periodic orbit. If $f^{\prime\prime\prime}(0)<0$, solutions in a small neighborhood of the origin are periodic orbits. Let $p(\tau)$ be a periodic orbit of \eqref{eq2.12}. We say $p(\tau)$ is non-degenerate if
$\dot p(\tau)$ is the only solution of the linearized equation of \eqref{eq2.12}. If we rewrite \eqref{eq2.12} as a first order system for $(p,p_\tau)$, namely,
\begin{equation}\label{eq2.13a}
\begin{pmatrix}
 p\\ p_\tau
\end{pmatrix}_\tau=\begin{pmatrix}
0 & 1\\
-1& 0
\end{pmatrix}\begin{pmatrix}
 p\\ p_\tau
\end{pmatrix}+\begin{pmatrix}
0 \\  -\frac 18 f'''(0)p^3 
\end{pmatrix}
\end{equation}
and linearize it along $(p,p_\tau)$, we obtain
\[\begin{pmatrix}
\tilde p\\\tilde p_1
\end{pmatrix}_\tau=\begin{pmatrix}
0 & 1\\
-1-\frac 38 f'''(0)p^2 & 0
\end{pmatrix}\begin{pmatrix}
\tilde p\\\tilde p_1 
\end{pmatrix}.\]
By Floquet theory, the non-degeneracy assumption of $p(\tau)$ implies 
\begin{enumerate}
\item[(ND)] The monodromy matrix generated by $\begin{pmatrix}
0 & 1\\
-1-\frac 38 f'''(0)p^2 & 0
\end{pmatrix}$ has $1$ as an eigenvalue with geometric multiplicity $1$.
\end{enumerate} 
Now we are ready to state our result.
\begin{MT}
Let $p(\tau)$ be a non-degenerate periodic solution of \eqref{eq2.12} with period $p$. If $f$ satisfies (A), then there exist $\ep_0\ll1$ and a (resonance) set $R_{p,\alpha,l}$ such that for every $\ep\in (0,\ep_0)\backslash R_{p,\alpha,l}$, \eqref{eq1.1} has a solution $u(x,t)$ even in $x$ and odd in $t$ satisfying 
\begin{equation}\label{eq2.14}
u(x,t+\frac{2\pi}{\sqrt{1+\ep^2}})=u(x,t)\ , \ u(x+\frac{p}{\ep\sqrt{1+\ep^2}},t)=u(x,t).
\end{equation}
Moreover, there exist two positive constants $C$ and $c$ independent of $\ep$ such that 
\begin{equation}\label{eq2.16}
\|u\|_{C^0_{x,t}}\leq C\ep\ , \ \|Q(\frac{u(x,\cdot)}{\ep}-p(\ep\sqrt{1+\ep^2}x))\|_{C^0_t}\leq Ce^{-\frac {c}{\ep}},\end{equation}
where $Q$ is the orthogonal projector onto the linear space $\oplus_{k=2}^{\infty}\{\mathbb{R}\sin{(k\sqrt{1+\ep^2}t)}\}$.
\end{MT}
The non-degeneracy assumption is standard in the continuation theory of periodic orbits, see \cite{C06} for example. We will use (ND) which is derived from the non-degeneracy of $p$ to match all coordinates except one of the time $p$-map for the perturbed system near the unperturbed orbit. The only missing direction will be recovered by the invariance of the Hamiltonian defined in \eqref{eq2.7}. The following analysis is based on \eqref{eq2.3} and \eqref{eq2.4}, which is obtained through the spatial dynamics formulation of \eqref{eq1.1}. Therefore, any result for \eqref{eq2.3} and \eqref{eq2.4} is reflected in \eqref{eq1.1} by swapping $x$ and $t$.

\section{Partial Normal Form Transformations}\label{P}
The plan of this section is to construct partial normal form transformations for \eqref{eq2.10} and \eqref{eq2.11}. Such transformations average out the $O(1)$ driving term $\tilde g$ to be $O(e^{-[\frac c\ep]})$. The transformations presented here is similar to the one in 
\cite{N84}. In general, as the example (an ODE system) shown in \cite{N84}, one cannot use such transformations to eliminate $\tilde g$ completely. Consequently, the exponentially small estimates in the main theorem is optimal. 

With slight abuse of notation, we will drop $\tilde {}$ signs in \eqref{eq2.10} and \eqref{eq2.11}, namely, we have
\begin{equation}\label{eq3.1}
\left\{\begin{aligned}
& v_{\tau\tau}=-\frac{1}{\om^2} v+ f(v,w,\ep),\\
& w_{\tau\tau}=\frac{J_\ep}{\ep^2}w+g(v,w,\ep).
\end{aligned}\right.
\end{equation}
Let $V=(v,v_\tau)$. In the first step, we set $w_1=w+\ep^2J_\ep^{-1}g(v,0,\ep)$, then
\begin{equation}\label{eq3.6}
\begin{aligned}
w_{1\tau\tau}=&\frac{J_\ep}{\ep^2}w_1+g(v,w_1-\ep^2J_\ep^{-1}g(v,0,\ep),\ep)-g(v,0,\ep)+\ep^2(J_\ep^{-1}g(v,0,\ep))_{\tau\tau}\\
=&\frac{J_\ep}{\ep^2}w_1+\Big(\int_0^1D_2g(v,p(w_1-\ep^2J_\ep^{-1}g(v,0,\ep)),\ep)\ dp\Big)w_1\\
&-\ep^2\Big(\int_0^1D_2g(v,p(w_1-\ep^2J_\ep^{-1}g(v,0,\ep)),\ep)\ dp\Big)L^{-1}g(v,0,\ep)\\
&+\ep^2(J_\ep^{-1}g(v,0,\ep))_{\tau\tau}\\
\triangleq&\frac{J_\ep}{\ep^2}w_1+g_1(v,w_1,\ep)w_1+\ep^2\bar g_1(V, w_1,\ep).
\end{aligned}
\end{equation}
We repeat such procedure $k$ times to obtain
\begin{equation}\label{eq3.7}
w_{k\tau\tau}=\frac{J_\ep}{\ep^2}w_k+g_k(v,w_k,\ep)w_k+\ep^{2k}\bar g_k(V,w_k,\ep).
\end{equation}
In the next step, we set $w_{k+1}=w_k+\ep^{2k+2}J_\ep^{-1}\bar g_k(V,0,\ep)$ to have
\begin{equation}\label{eq3.8}
\begin{aligned}
\p_{\tau\tau}w_{k+1}=&\frac{J_\ep}{\ep^2}w_{k+1}+g_k(v,w_k,\ep)w_{k+1}+\ep^{2k}\big(\int_0^1D_2\bar g_k(V,pw_k,\ep)\ dp\big)w_{k+1}\\
&-\ep^{2k+2}\big(\int_0^1D_2\bar g_k(V,pw_k,\ep)\ dp\big)J_\ep^{-1}\bar g_k(V,0,\ep)\\
&-\ep^{2k+2}g_k(v,w_k,\ep)J_\ep^{-1}\bar g_k(V,0,\ep)+\ep^{2k+2}(J_\ep^{-1}\bar g_k(V,0,\ep))_{\tau\tau}\\
\triangleq&\frac{J_\ep}{\ep^2}w_{k+1}+g_{k+1}(v,w_{k+1},\ep)w_{k+1}+\ep^{2k+2}\bar g_{k+1}(v,w_{k+1},\ep).
\end{aligned}
\end{equation}
 
To obtain estimates on $(g_k,\bar g_k)$, we complexify $H^1$ to $H^1\oplus iH^1$. We define the complex neighborhood of the real axis
\[S_K\triangleq\{z\in\mathbb C\big||\Im z|<K\}.\]
Since we will construct bounded orbits in the phase space, we assume
\begin{equation}\label{eq3.9}
(\|f\|+\|g\|)_{C^2(S_K\times S_K\times [0,\ep_0))}+\|u\|_{C^1}\leq C(K).
\end{equation}
Here we choose $K$ to be large enough so that any prescribed non-degenerate periodic orbit of \eqref{eq2.12} have magnitude less than $K$.
\begin{prop}\label{prop3.1}
Given $K>0$, there exist $C(K)$, $K_1\gg 1$ and $0<\ep_0\ll 1$ such that for every $\ep\in[0,\ep_0)$,
\begin{equation}\label{eq3.10}
\begin{aligned}
&\|g_k\|_{C^0(S_{K-\ep k K_1}\times S_{K-\ep k K_1}\times [0,\ep_0))}\leq (2-\frac{1}{K_1^{k-1}})C(K),\\
&\|\ep^{2k}\bar g_k\|_{C^0(S_{K-\ep k K_1}\times S_{K-\ep k K_1}\times [0,\ep_0))}\leq \frac{ C(K)}{K_1^{k-1}}.
\end{aligned}
\end{equation}
\end{prop}
\begin{proof}
We will prove \eqref{eq3.10} inductively. For $k=1$, by definition in \eqref{eq3.6}, we have
\[\|g_1\|_{C^0(S_{K-\ep  K_1}\times S_{K-\ep  K_1}\times [0,\ep_0))}\leq \|Dg\|_{C^0(S_{K}\times S_{K}\times [0,\ep_0))}\leq C(K),\]
and
\[\begin{aligned}
&\|\ep^2\bar g_1\|_{C^0(S_{K-\ep k K_1}\times S_{K-\ep k K_1}\times [0,\ep_0))}\\
\leq & \ep^2\|J_\ep^{-1}\|\|g\|_{C^2}(\|g\|_{C^0}+\|v_x\|_{C^0}^2+\|v\|_{C^0}+\|f\|_{C^0})\leq C(K).
\end{aligned}\]
Assuming \eqref{eq3.10} holds for $k-1$, i.e., 
\begin{equation}\label{eq3.11}
\begin{aligned}
&\|g_{k-1}\|_{C^0(S_{K-\ep (k-1) K_1}\times S_{K-\ep (k-1) K_1}\times [0,\ep_0))}\leq (2-\frac{1}{K_1^{k-2}})C(K),\\
&\|\ep^{2(k-1)}\bar g_{k-1}\|_{C^0( S_{K-\ep (k-1) K_1}\times S_{K-\ep (k-1) K_1}\times [0,\ep_0))}\leq \frac{ C(K)}{K_1^{k-2}}.
\end{aligned}
\end{equation}
By definition of $(g_k,\bar g_k)$ in \eqref{eq3.8} and the Cauchy integral formula,
\[\begin{aligned}
\|g_{k}\|_{C^0(S_{K-\ep k K_1}\times S_{K-\ep k K_1}\times [0,\ep_0))}
\leq & \|g_{k-1}\|_{C^0( S_{K-\ep (k-1) K_1}\times S_{K-\ep (k-1) K_1}\times [0,\ep_0))}\\
&+\ep^2\Big|\frac{1}{2\pi i}\oint_{B_{\ep K_1}(\mathbb C)}\frac{\ep^{2k-2}g_{k-1}(\ep,v+z,w_{k-1}+z)}{z^2}\ dz\Big|\\
\leq & (2-\frac{1}{K_1^{k-2}})C(K)+\ep^2\frac{C(K)}{K_1^{k-2}}\frac{1}{K_1\ep}\leq (2-\frac{1}{K_1^{k-1}})C(K),
\end{aligned}\]
and
\[\begin{aligned}
&\|\ep^{2k}\bar g_k\|_{C^0(S_{K-\ep k K_1}\times S_{K-\ep k K_1}\times [0,\ep_0))}\\
\leq & 2\ep^2\frac{2C(K)}{K_1\ep}\|J_\ep^{-1}\|\frac{ C(K)}{K_1^{k-2}}+2\ep^2\frac{(|v_x|^2+|v|+C(K)) C(K)}{K_1^{k-2}(K_1\ep)^2}\leq\frac{ C(K)}{K_1^{k-1}}.
\end{aligned}\]
Thus, the proof is completed.
\end{proof}
The proposition implies we can perform the partial normal form transformations $[\frac c\ep]$ times, where $c$ is possibly small but independent of $\ep$. Then the driving term $\ep^{2[\frac{c}{\ep}]}\bar g_{[\frac c\ep]}$ is exponentially small in $\ep$. With slight abuse of notation, we write the transformed system as
\begin{equation}\label{eq3.12}
\left\{\begin{aligned}
& v_{\tau\tau}=-\frac{1}{\om^2} v+ f(v,w,\ep),\\
& w_{\tau\tau}=\frac{J_\ep}{\ep^2}w+g(v,w,\ep)+\bar g(v,v_\tau,w,\ep),
\end{aligned}\right.
\end{equation}
By shrinking $K-cK_1$ a little bit, we can assume
\begin{equation}\label{eq3.13}
(\|f\|+\|g\|+e^{[\frac c\ep]}\|\bar g\|)_{C^2(S_{K-2c K_1}\times S_{K-2c K_1}\times [0,\ep_0))}\leq C.
\end{equation}
The partial normal form transformations can be thought as an averaging procedure, which makes $\{w=0\}$ to be almost invariant up to an error of $O(e^{-[\frac c\ep]})$. 
\begin{rem}
If $f\in C^k$, where $k\geq 5$, we can obtain $\{w=0\}$ is almost invariant up to $O(\ep^{2(k-5)})$.
\end{rem}

\section{Proof}\label{proof} 
In this section, we give the proof for the Main Theorem. To find periodic solutions of \eqref{eq3.12}, we study their time-$p$ map around the periodic orbit of \eqref{eq2.12} . The main difficulty is the small divisor problem, which is overcome by carefully choosing a non-resonance set and running a Nash-Moser iteration argument. We adopt and modify the strategy introduced in \cite{BB,BB10} to handle the singular parameter of our problem. We start by rewriting \eqref{eq3.12} in a slightly different form
\begin{equation}\label{eq6.1}
\left\{\begin{aligned}
&\begin{pmatrix}
v \\ v_1
\end{pmatrix}_\tau=\begin{pmatrix}
0 & \frac 1\om\\
-\frac 1\om & 0
\end{pmatrix}\begin{pmatrix}
v \\ v_1
\end{pmatrix}+\begin{pmatrix}
0 \\ \om f(v,w,\ep)
\end{pmatrix},\\
&(J_\ep-\ep^2\p_{\tau\tau})w+\ep^2 \big((g(v,w,\ep)+\bar g(v,v_\tau,w,\ep))\big)=0,
\end{aligned}\right.
\end{equation}
where we recall that $\om=\sqrt{1+\ep^2}$ and $J_\ep=\p_{xx}+\frac{1}{1+\ep^2}$.

Let $P_\star$ be a non-degenerate (in the sense of (ND) in Section \ref{SM}) periodic orbit of \eqref{eq2.13a} with period $p$. We choose $P_0=P_\star(0)\in P_\star$ and use $v,v^\bot$ to denote $\dot P_\star(0)$ and $DH_\star(P_0)$, respectively. The invariance of $H_\star$ (defined in \eqref{eq2.14}) for \eqref{eq2.13a} implies $v$ is perpendicular to $v^\bot$. Let $\mathcal V(V(0);w,\ep)$ be the time $p$-map of the first equation of \eqref{eq6.1} with initial condition at $V(0)$ and parameters $(w,\ep)$. We look for periodic solutions of \eqref{eq6.1} as follows:
\begin{enumerate}
\item[(1)] Given any $p$-periodic function $V(\cdot)$ such that there exists $V(0)\in V(\cdot)$ with
\[V(0)=P_0+\delta_1v^\perp.\] 
We identify a resonance set $R_{p,\alpha,l}$ such that for $\ep\in (0,\ep_0)\backslash R_{p,\alpha,l}$, we obtain a  solution $w(V,\ep)$ that satisfies the second equation of \eqref{eq6.1}. 
\item[(2)] Plugging such $w$ into the first equation of \eqref{eq6.1} to find $V(\ep)$ such that
\begin{equation}\label{U1}
P_v\big(\mathcal V(P_0+\delta_1(\ep)v^\bot;w(V(\ep),\ep),\ep)-P_0\big)=0,\end{equation}
where $P_v$ is the orthogonal projection onto $\R v$.
\item[(3)] Use the invariance of $H$ defined in \eqref{eq2.7} to recover the missing direction of $V(\ep)$ and thus conclude $V(\ep)$ satisfies the first equation of \eqref{eq6.1}.
\end{enumerate}
The rest of this section is split into $4$ subsections. In subsection \ref{function space}, we introduce the function space and the linear operator $\mathcal L$ associated to the $w$-equation. In the next subsection, we identify the resonance set and analyze the bounds on the inverse of $\mathcal L$. In subsection \ref{NM}, we go through the Nash-Moser iteration to obtain solution for the $w$-equation. Finally, we solve the $V$-equation.

\subsection{Function Space $\Gamma_s$ and Linear Operator $\mathcal L$}\label{function space} 
The function space we will be working on for the $p$-equation is the Sobolev space involving space and time.
We restrict ourselves to solutions that are even in $\tau$. Define the space
\begin{equation}\label{S}
\begin{aligned}
\Gamma_s=\{w\in H^1([0,p],QH^s)\big| w(x,t)=\sum_{\substack{|k|\geq 2\\ j\in\mathbb Z}}&w_{j,k}e^{i\frac{2\pi j}{p}\tau+i k x},\\
&w_{j,k}=w_{-j,k}=-w_{j,-k}\},\end{aligned}
\end{equation}
which is equipped with the standard space-time Sobolev norm
\[\|w\|_s^2=\sum_{|k|\geq 2,j}(1+|j|)^2|k|^{2s}|w_{j,k}|^2.\]
Let $\Pi_N$ and $\Pi_N^c$ be the projection operators (along $x$-direction) given by 
\[\widehat{\Pi_Nf}(k)=\left\{\begin{aligned}
\hat f(k)\ , \ &|k|\leq N\\
0\ , \ & |k|>N
\end{aligned}
\right.
\ , \ \widehat{\Pi_{N}^cf}(k)=\left\{\begin{aligned}
0\ , \ &|k|\leq N\\
\hat f(k)\ , \ & |k|>N
\end{aligned}
\right.
.\]
By standard argument, we have
\begin{prop}
For any $N\geq1$, $0\leq m_1<m_2$ and $m=(1-t)m_1+tm_2$,
\begin{eqnarray}\label{lp1}
&&\|\Pi_Nh\|_{m_2}\leq N^{m_2-m_1}\|h\|_{m_1}\ , \ \forall h\in \Gamma_{m_1},\\\label{lp2}
&& \|\Pi_N^ch\|_{m_1}\leq N^{-(m_2-m_1)}\|h\|_{m_2}\ , \ \forall h\in \Gamma_{m_2}.
\end{eqnarray}
\end{prop}

Let
\begin{equation}\label{eq6.6}
\tilde g(V,w,\ep)=g(v,w,\ep)w+\bar g(V,w,\ep),
\end{equation}
where $(g,\bar g)$ are from \eqref{eq6.1} and $V=(v,v_1)$. Let $H^1_\tau$ be the function space that consists of $H^1$ functions which are $p$-periodic and even in $\tau$ endowed with the inner product 
\begin{equation}\label{inner}
\langle u,v\rangle=\int_{0}^{p} u_\tau v_\tau+(\frac{1}{\pi}\int_{-\pi}^{\pi}D_w\tilde g(V,w,\ep)\ dx+C)uv\ d\tau \ , \ C\gg1.\footnote{The function $D_w\tilde g(V,w,\ep)$ is even in $x$. This is because $\tilde g$ is odd in $p$ and $V$ enters $\tilde g$ as $V\sin{y}$.}
\end{equation}
It is easy to verify that $(-\p_{\tau\tau}+\frac{1}{\pi}\int_{-\pi}^{\pi}D_w\tilde g(V,w,\ep)\ dx+C)^{-1}$ is a compact, positive and self-adjoint operator on $H_\tau^1$ under the inner product $\langle \cdot,\cdot\rangle$ introduced above. It implies the eigenvalues $\lambda_j$ of linear operator $-\p_{xx}+\frac 1\pi \int_{-\pi}^{\pi}D_w\tilde g(P_\star,p,\ep)\ dx$ are simple, bounded from below and grow like $O(j^2)$. Moreover, there exists an orthonormal basis $\{\phi_j(\tau)\}_{j\geq1}$ (with respect to the inner product in \eqref{inner}) such that
\[\big(-\p_{\tau\tau}+\frac 1\pi \int_{-\pi}^{\pi}D_w\tilde g(V,w,\ep)\ dx\big)\phi_j(\tau)=\lambda_j\phi_j(\tau).\]
Let $\ep_{k,j}>0$ be the solution of 
\[-k^2+\frac{1}{1+\ep_{k,j}^2}+\ep^2_{k,j}\lambda_j=0,\]
where $|k|\geq 2$. It is easy to calculate for $j\gg1$,
\begin{equation}\label{eq6.8}
\ep_{k,j}=\frac{p}{2\pi}|\frac{k}{j}|+O(|\frac{k}{j^3}|).
\end{equation}
Given any $(V,w)$ and $2\leq N_1\leq N_2$, we let
\begin{equation}\label{Set}
 R_{V,w}^{N_1,N_2}=\bigcup_{N_1\leq|k|\leq N_2,j}(\ep_{k,j}-\frac{ |k|^\alpha}{j^l},\ep_{k,j}+\frac{ |k|^\alpha}{j^l}),
 \end{equation}
where $2\leq 2+\alpha<l<3$. 

We define the nonlinear map $\mathcal F$ and its linearization $\mathcal L$ as 
\begin{eqnarray}\label{F}
&&\mathcal F(V,w,\ep)\triangleq (J_\ep-\ep^2\p_{\tau\tau})w+\ep^2\tilde g(V,w,\ep),\\\label{L}
&&\mathcal L(V,w,\ep)\triangleq J_\ep+\ep^2(-\p_{\tau\tau}+D_w\tilde g(V,w,\ep)).
\end{eqnarray}
Thus, finding a solution for the $w$-equation is equivalent to solving 
\begin{equation}\label{eq6.9.1}
\mathcal F(V,w,\ep)=0.\end{equation}
We will find a solution of \eqref{eq6.9.1} through a sequence of approximate solutions $\{w_i\}_{i\geq1}$, where each $w_i\in\Pi_{N_i}\Gamma_s$ satisfies
\begin{equation}\label{eq6.9.2}
\Pi_{N_i}\mathcal F(U,w_i,\ep)=0 \ , \ \lim_{i\to\infty}N_i=\infty.
\end{equation}
We will solve \eqref{eq6.9.2} for each $i$ by a contraction mapping argument. It involves inverting the linear operator $\Pi_{N_i}\mathcal L$, which is analyzed in the next subsection.

\subsection{Resonance Set $R_{p,\alpha,l}$ and Bounds on $(\Pi_N\mathcal L)^{-1}$} In this subsection, we identify the resonance set $R_{p,\alpha,l}$ so that for $\ep\notin R_{p,\alpha,l}$, each $(\Pi_N\mathcal L(U,w_i,\ep))^{-1}$ has a good enough bound to complete the Nash-Moser iteration. To simplify our notations, we let
\begin{equation}\label{eq6.10}
\gamma=l-\alpha-2\in(0,1).
\end{equation}
 
\begin{prop}\label{prop6.2}
For any $s>\frac{\gamma+1}{2}$, $N\geq 2$, $w\in\Gamma_{s+\frac{\gamma}{2}}$ and $\ep\in(0,\ep_0)\backslash R_{V,w}^{2,+\infty}$ and $h\in \Pi_N\Gamma_s$, we have
\begin{equation}\label{eq6.11}
\|(\Pi_N\mathcal L(V,w,\ep))^{-1}h\|_{s-\gamma}\leq C\ep^{-(l-1)}\|h\|_s,
\end{equation}
where $C$ is independent of $\ep$ and $h$. Consequently,
\begin{equation}\label{eq6.12}
\|(\Pi_N\mathcal L(V,w,\ep))^{-1}h\|_{s}\leq CN^\gamma\ep^{-(l-1)}\|h\|_s.
\end{equation}
\end{prop}
\begin{proof}
We first rewrite $\Pi_N\mathcal L(V,w,\ep)$ as
\[\Pi_N\mathcal L(V,w,\ep)=\mathcal L_1(V,w,\ep)+\ep^2A(V,w,\ep),\]
where
\begin{equation}\label{A}
\begin{aligned}
&\mathcal L_1(V,w,\ep)=J_\ep+\ep^2(-\p_{xx}+\frac{1}{\pi}\int_{-\pi}^{\pi}\Pi_ND_w\tilde g(V,w,\ep)\ dx),\\
&A(V,w,\ep)=\Pi_N\big(D_w\tilde g(V,w,\ep)-\frac{1}{\pi}\int_{-\pi}^{\pi}D_w\tilde g(V,w,\ep)\ dx\big).
\end{aligned}\end{equation}
For $\ep\in(0,\ep_0)\backslash R_{V,w}^{2,+\infty}$, $\mathcal L_1$ is invertible and we define $|\mathcal L_1|^{-\frac 12}$ as
\[|\mathcal L_1|^{-\frac 12}h=\sum_{2\leq |k|\leq N}\sum_{j\geq1}\frac{h_{j,k}}{\sqrt{|-k^2+\frac{1}{1+\ep^2}+\ep^2\lambda_j|}}\phi_j(\tau)e^{ikx}\]
for $h=\displaystyle\sum_{2\leq |k|\leq N}\sum_{j\geq1}h_{j,k}\phi_j(\tau)e^{ikx}$. Then we can write
\[\Pi_N\mathcal L=|\mathcal L_1|^{\frac 12}\big(|\mathcal L_1|^{-\frac 12} \mathcal L_1 |\mathcal L_1|^{-\frac 12}+\ep^2|\mathcal L_1|^{-\frac 12}A(V,w,\ep)|\mathcal L_1|^{-\frac 12}\big)|\mathcal L_1|^{\frac 12}.\]
It is clear that 
\[\big\||\mathcal L_1|^{-\frac 12}\mathcal L_1|\mathcal L_1|^{-\frac 12}\big\|_{L(\Pi_N\Gamma_s,\Pi_N\Gamma_s)}=1.\]
By \eqref{eq6.8} and \eqref{Set}, we have
\[\big\||\mathcal L_1|^{-\frac 12}\big\|_{L(\Pi_N\Gamma_s,\Pi_N\Gamma_{s-\frac{\gamma}{2}})}\leq C\sqrt{\frac{k^{\alpha+1+\gamma}}{j^{l-1}}}\leq C\ep^{-\frac{l-1}{2}}.\]
Therefore, we finish the proof if we can show $\ep^2|\mathcal L_1|^{-\frac 12}A(V,w,\ep)|\mathcal L_1|^{-\frac 12}$ is a small bounded linear operator on $\Pi_N\Gamma_{s-\frac{\gamma}{2}}$. For any $h(x,y)\in \Pi_N\Gamma_{s-\frac \gamma2}$,
\[\begin{aligned}
&|\mathcal L_1|^{-\frac 12}A(V,w,\ep)|\mathcal L_1|^{-\frac 12}h\\
=&|\mathcal L_1|^{-\frac 12}\Pi_N\big(\sum_{k\in\mathbb Z}A_{k}(x,\ep)e^{ikx}\big)
\big(\sum_{2\leq |k|\leq N}(\sum_{j\geq1}\frac{h_{k,j}}{\sqrt{|-k^2+\frac{1}{1+\ep^2}+\ep^2\lambda_j|}}\phi_j(\tau))e^{ikx}\big)\\
=&|\mathcal L_1|^{-\frac 12}\Big[\sum_{|k'|\leq N}\Big(\sum_{2\leq |k|\leq N}\big(\sum_{j\geq1}
\frac{h_{k,j}}{\sqrt{|-k^2+\frac{1}{1+\ep^2}+\ep^2\lambda_j|}} A_{k'-k}(x,\ep)\phi_j(\tau)
\big)\Big)e^{ik'y}\Big],
\end{aligned}\] 
which implies
\[
\begin{aligned}
&\big(|\mathcal L_1|^{-\frac 12}A(V,w,\ep)|\mathcal L_1|^{-\frac 12}h\big)_{k'}\\
=&\sum_{2\leq |k|\leq N}\big(\sum_{j\geq1}
\frac{h_{k,j}}{\sqrt{|-k'^2+\frac{1}{1+\ep^2}+\ep^2\lambda_j|}\sqrt{|-k^2+\frac{1}{1+\ep^2}+\ep^2\lambda_j|}}\\
&\hspace{9.5cm}A_{k'-k}(x,\ep)\phi_j(\tau)
\big).
\end{aligned}\]
Let $m_k:=\displaystyle\min_{j\geq1}{|-k^2+\frac{1}{1+\ep^2}+\ep^2\lambda_j|}$. There exists a unique $j(k)$ such that
\[m_k:=\min_{j\geq1}{|-k^2+\frac{1}{1+\ep^2}+\ep^2\lambda_j|}=|-k^2+\frac{1}{1+\ep^2}+\ep^2\lambda_{j(k)}|\geq \frac{C\ep^{l-1}}{|k|^\gamma},\]
where $C$ is independent of $k,j$ and $\ep$. Note that the norm induced by the inner product defined in \eqref{inner} is equivalent to the standard $H^1$ Sobolev norm. Consequently,
\begin{equation}\label{eq6.15}
\|\big(|\mathcal L_1|^{-\frac 12}A(V,w,\ep)|\mathcal L_1|^{-\frac 12}h\big)_{k'}\|_{H_\tau^1}
\leq \sum_{|k|\geq 2}\|\frac{1}{\sqrt{m_{k'}m_k}}A_{k'-k}\|_{H_\tau^1}\|h_k\|_{H_\tau^1},
\end{equation}
where $h_k=h_k(x)=\displaystyle\sum_{j\geq1}h_{k,j}\phi_j(\tau)$. The definition of $A$ in \eqref{A} implies $A_0=0$. Now we claim that for $k\ne k'$,
\begin{equation}\label{eq6.16}
\frac{1}{\sqrt{m_{k'}m_k}}\leq C\ep^{-(l-1)}|k'-k|^{\gamma}.
\end{equation}
Without loss of generality, we assume $0<k'<k$. For each $\ep\in (0,\ep_0)\backslash R_{V,w}^{2,+\infty}$, there exist $j(k')$ and $j(k)$ such that
\begin{equation}\label{eq6.17}
\ep\in (\frac{k'}{j(k')},\frac{k'}{j(k')-1})\cap (\frac{k}{j(k)},\frac{k}{j(k)-1}).
\end{equation}
If $2k'\leq k$, then (recall that $\gamma=l-(\alpha+2)$)
\[\sqrt{m_{k'}m_k}\geq \sqrt{\frac{k'^{\alpha+1}}{j(k')^{l-1}}\frac{k^{\alpha+1}}{j(k)^{l-1}}} \geq C\frac{\ep^{(l-1)}}{(k'k)^{\frac{\gamma}{2}}}\geq C\frac{\ep^{l-1}}{|k-k'|^\gamma}.\]
For $k'<k<2k'$, we note
\[|k-k'-\ep(j(k)-j(k'))|\geq C\frac{|k-k'|^\alpha}{|j(k)-j(k')|^{l-1}}.\]
From \eqref{eq6.17}, one has
\[0<j(k)-j(k')<\frac{k}{\ep}+1-\frac{k'}{\ep}\leq \frac C\ep|k-k'|.\]
Combining with the above inequality, one has
\[|k-k'-\ep(j(k)-j(k'))|\geq \frac{C|k-k'|^\alpha}{\ep^{-(l-1)}|k-k'|^{l-1}}=\frac{C}{\ep^{-(l-1)}|k-k'|^{\gamma+1}},\]
which implies
\[\max\{|k-\ep j(k)|,|k'-\ep j(k')|\}\geq\frac{C\ep^{l-1}}{2|k-k'|^{\gamma+1}}.\]
If $|k'-\ep j(k')|\geq \frac{C\ep^{l-1}}{2|k-k'|^{\gamma+1}}$, then
\[\sqrt{m_km_{k'}}\geq\sqrt{\frac{C\ep^{l-1}k'}{2|k-k'|^{\gamma+1}}\frac{\ep^{l-1}}{k^\gamma}}\geq
\sqrt{\frac{C\ep^{l-1}k^{1-\gamma}\ep^{l-1}}{4|k-k'|^{\gamma+1}}}\geq\sqrt{\frac{C\ep^{2(l-1)}}{4|k-k'|^{2\gamma}}}=\frac{C\ep^{l-1}}{|k-k'|^\gamma}.
\]
The case for $|k-\ep j(k)|\geq \frac{C\ep^{l-1}}{2|k-k'|^{\gamma+1}}$ is similar. Thus, we obtain \eqref{eq6.16}. Since $A$ is smooth in $U$ and $p$, we use \eqref{eq6.15} to conclude
\[
\|\ep^2\big(|\mathcal L_1|^{-\frac 12}A(V,w,\ep)|\mathcal L_1|^{-\frac 12}h\big)_{k'}\|_{H_\tau^1}
\leq  C\ep^{3-l}\sum_{|k|\geq 2}\||k-k'|^\gamma A_{k'-k}\|_{H_\tau^1}\|h_k\|_{H_\tau^1},\]
which implies (thanks to the fact $s>\frac 12$ and Sobolev embedding)
\[\ep^2\big\||\mathcal L_1|^{-\frac 12}A(V,w,\ep)|\mathcal L_1|^{-\frac 12}h\|_{s-\frac\gamma2}\leq C\ep^{3-l}\|A(V,w,\ep)\|_{s+\frac \gamma 2}\|h\|_{s-\frac \gamma 2}\leq \frac 12 \|h\|_{s-\frac \gamma 2}.\]
Therefore, we finish the proof of \eqref{eq6.11}. Finally, we note $h$ is finite dimensional in $x$. Thus, we can apply \eqref{lp1} to obtain \eqref{eq6.12}.
\end{proof}
\begin{rem}
From the above analysis, one can see there is a trade-off among the loss of spatial ($x$-variable) regularity, temporal ($\tau$-variable) and the bound on $(\Pi_N\mathcal L)^{-1}$. If we choose to sacrifice the temporal regularity, we can bound $(\Pi_N\mathcal L)^{-1}$ like $O(\frac{1}{\ep^{\alpha+1}})$, which is smaller than $O(\frac{1}{\ep^{l-1}})$. However, the solution of the $p$-equation $w(V,\ep)$ has less temporal regularity than $U$, which makes the solution scheme for $U$ incoherent. Since the $V$-equation is finite dimensional in $y$, sacrificing the spatial regularity will not cause any trouble. Even though we get a worse bound $O(\frac{1}{\ep^{l-1}})$ in this case, it still can be controlled by the $\ep^2$ term in the nonlinear part of $\mathcal F$.
\end{rem}

In order to make sure $\Pi_{N_i}\mathcal L(V,w_{i-1},\ep)$ satisfies a similar bound as in \eqref{eq6.12} for each $i$, we need to control $\|w_i-w_{i-1}\|_s$ and remove $\ep$ out of a set that is slightly larger than $R_{U,0}^{2,+\infty}$. We choose $\{N_i\}_{i\geq1}$ to $\ep$ by setting
\begin{equation}\label{eq6.25}
\ep^{-1}<N_1=[\frac 12(\ep^{-1}+\ep^{-2})]<\ep^{-2}\ , \ N_{i+1}=N_i^2,
\end{equation}
where $[\cdot]$ denotes the largest integer that is less than or equal to the number inside the bracket. We will skip writing the dependence of $\Pi_{N}\mathcal L$ on $(U,\ep)$ in the rest of this section, i.e., $\Pi_{N}\mathcal L(w)=\Pi_N\mathcal L(V,w,\ep)$.

\begin{Lemma}\label{le6.3}
Suppose $\|w_i-w_{i-1}\|_s\leq N_{i}^{-\sigma}$ and $\|\p_\ep w_i\|_s\leq\frac 12$, where $0<\gamma+l<\sigma$, there exists a set $R_{p,\alpha,l}\subset(0,\ep_0)$ with zero measure as $\ep_0\to0$ such that for any $\ep\in(0,\ep_0)\backslash R_{p,\alpha,l}$, 
\begin{equation}\label{eq6.26}
\|(\Pi_{N_i}\mathcal L)^{-1}(w_{i-1})\|_{L(\Gamma^i_s,\Gamma^i_{s})}\leq \frac{CN_i^\gamma}{\ep^{l-1}}.
\end{equation}
\end{Lemma}
\begin{proof}
We first define the parameter set for any given $w$ ($V$ is fixed) 
\begin{equation}\label{para}
S_N(w)\triangleq \{\ep\Big| \|(\Pi_{N}\mathcal L(w))^{-1}\|_{L(\Pi_N\Gamma_s,\Pi_N\Gamma_s)}\leq \frac{CN^\gamma}{\ep^{l-1}}\},
\end{equation}
where $C$ is the same one as in \eqref{eq6.11}. The above definition can also be understood as none of the eigenvalues of $\Pi_N\mathcal L(w)$ is in the interval $[-\frac{\ep^{l-1}}{CN^\gamma},\frac{\ep^{l-1}}{CN^\gamma}]$.

We will only prove the result for $i=1,2$ and the general result follows similarly. For $i=1$ and $\ep\in(0,\ep_0)\backslash R_{U,0}^{2,+\infty}$, the result is proved in Proposition \ref{prop6.2}. For $i=2$, we need to estimate the measure of $\big(\cup_{N=2}^{N_2} S_N(w_1)\big)^c\backslash \big(\cup_{N=2}^{N_1} S_N(0)\big)^c$. Note that
\[\begin{aligned}
R_1\triangleq&\big(\bigcup_{N=2}^{N_2} S_N(w_1)\big)^c\backslash \big(\bigcup_{N=2}^{N_1} S_N(0)\big)^c\\
\subset &\big(\bigcup_{2\leq N\leq N_1}(S_N^c(w_1)\cap S_N(0)\big)\bigcup \big(\bigcup_{N>N_1}S_N^c(w_1)\big).\end{aligned}
\]
By definition, for each $N\leq N_1$, we have
\[S_N^c(w_1)\cap S_N(0)\subset \{\ep\big| \mathcal L(0)\ \text{has an eigenvalue in}\ [-\frac{\ep^{l-1}}{CN^{\gamma}}-N_1^{-\sigma},\frac{\ep^{l-1}}{CN^{\gamma}}+N_1^{-\sigma}]\}.\]
Since $\ep\in(0,\ep_0)\backslash R_{U,0}^{2,+\infty}$, 
\[\begin{aligned}
S_N^c(w_1)\cap S_N(0)\subset \{\ep\big| \mathcal L(0)\ \text{has an eigenvalue in}\ &[-\frac{\ep^{l-1}}{CN^{\gamma}}-N_1^{-\sigma},-\frac{\ep^{l-1}}{CN^{\gamma}}] \\
&\ \text{or}\ [\frac{\ep^{l-1}}{CN^{\gamma}},\frac{\ep^{l-1}}{CN^{\gamma}}+N_1^{-\sigma}]\}.
\end{aligned}\]
By \eqref{eq6.8} and the assumption $\|\p_\ep w_i\|_s\leq \frac 12$, 
\[\frac{d}{d\ep}(-k^2+\frac{1}{1+\ep^2}+\ep^2\lambda_j)\big|_{\ep=\ep_{N,j}}\leq -|N|j.\] Consequently,  the measure
\[\begin{aligned}
|R_1|\triangleq&|\big(\bigcup_{2\leq N\leq N_1}(S_N^c(w_1)\cap S_N(0)\big)\big(\bigcup_{N>N_1}S_N^c(w_1)\big)|\\
\leq &C'\big(\sum_{2\leq N\leq N_1}\sum_{j=\frac{N}{\ep_0}}^\infty\frac{N_1^{-\sigma}}{Nj}+
\sum_{N\geq N_1}\sum_{j=\frac{N}{\ep_0}}^\infty\frac{\ep^{l-1}}{CN^\gamma}\frac{1}{Nj}\big)\\
\leq & C'\Big(\sum_{2\leq k\leq N_1}\sum_{j=\frac{k}{\ep_0}}^\infty \frac{(\frac {N}{j})^{l-1}}{kj N^\gamma}
+\sum_{N\geq N_1}\sum_{j=\frac{N}{\ep_0}}^\infty\frac{\frac{N}{j}^{l-1}}{CN^\gamma}\frac{1}{Nj}\Big)\\
\leq & C'\big(\sum_{2\leq N\leq N_1}\frac{\ep_0^{l-1}}{NN_1^{\sigma-(l-1)}}+\sum_{N\geq N_1}\frac{\ep_0^{l-1}}{CN^{1+\gamma}}\big)\leq C'\frac{\ep_0^{l-1}}{N_1^\gamma},
\end{aligned}\]
where we use the fact $N_1^{-1}\leq \ep$ and $l+\gamma<\sigma$ to obtain the second inequality.
Now we update the resonance set to be $R_{V,0}^{2,+\infty}\cup R_1$. By iterating the above procedure, one can obtain \eqref{eq6.26} for all $i$ on $(0,\ep_0)\backslash R_{p,\alpha,l}$, where
\[R_{p,\alpha,l}=R_{V,0}^{2,+\infty}\bigcup (\bigcup_{i=1}^\infty R_i)\ , \ R_i=\big(\bigcup_{N=2}^{N_i} S_N(w_i)\big)^c\backslash \big(\bigcup_{N=2}^{N_{i-1}} S_N(w_{i-1})\big)^c.\]
The only thing left is to prove the measure estimate. It is easy to see that the measure of $R_{p,\alpha,l}$ can bounded above by
\[(\sum_{|k|=2}^\infty\sum_{j=\frac{|k|}{\ep_0}}^\infty\frac{2|k|^\alpha}{j^l})(1+\sum_{i=1}^\infty\frac{C'}{N_i^\gamma})\leq C\sum_{|k|=2}^\infty
|k|^\alpha (\frac {k}{\ep_0})^{1-l}\leq C'\ep_0^{l-1}=\ep_0 O(\ep_0^{l-2}),\]
where we use the fact $2+\alpha<l<3$ to obtain the last step. The proof is completed.
\end{proof}

\subsection{The Nash-Moser Iteration}\label{NM}
In this subsection, we apply the Nash-Moser iteration technique to find a convergent sequence $\{w_i\}_{i\geq1}\in \Pi_{N_i}\Gamma_s$ such that each $w_i$ satisfies
\begin{equation}\label{F_i}
\Pi_{N_i}\mathcal F(w_i)=0.
\end{equation}
To simplify the following presentation, we introduce some notations
\begin{eqnarray}\label{notation1}
&& \tilde w_i=w_i-w_{i-1}\ , \ \mathcal L_i=\Pi_{N_i}\mathcal L(w_{i-1})\ , \ \Gamma_{s}^i=\Pi_{ N_i}\Gamma_{s},\\\label{notation2}
&& r_i=(\Pi_{N_i}-\Pi_{N_{i-1}})\mathcal F(w_{i-1})=(\Pi_{N_i}-\Pi_{N_{i-1}})\ep^2\tilde g(w_{i-1}),\\\label{notation3}
&& R_i(\tilde w_i )=\Pi_{N_i}(\mathcal F(w_i)-\mathcal F(w_{i-1})-\mathcal L_i\tilde w_i),
\end{eqnarray}
where $p_0=0$ and $N_0=2$. Here each $r_i$ and $R_i$ satisfy the tame property. More precisely, for $s\leq \bar s$ and $\|w_i\|_s\leq 1$, we have
\begin{eqnarray}\label{tame1}
&& \|r_i\|_{\bar s}+\|D_{V,w}r_i\|_{\bar s}\leq C(\bar s)(1+\|w_{i-1}\|_{\bar s}),\\\label{tame2}
&& \|R_i(\tilde w_i)\|_{\bar s}\leq C(\bar s)(\|w_{i-1}\|_{\bar s}\|\tilde w_i\|_s^2+\|\tilde w_i\|_s\|\tilde w_i\|_{\bar s}).
\end{eqnarray}
We also need for any $s>\frac 12, \bar s\geq 0$ and $u_{1,2}\in \Gamma_s\cap\Gamma_{\bar s},$
\begin{equation}\label{tame3}
\|u_1 u_2\|_{\bar s}\leq C(\bar s)(\|u_1\|_s\|u_2\|_{\bar s}+\|u_1\|_{\bar s}\|u_2\|_s).
\end{equation}
We refer the proof of \eqref{tame1}--\eqref{tame3} to Sections $2$ and $4$ in \cite{BB}. Using notations in \eqref{notation1}, \eqref{notation2} and \eqref{notation3}, one can solve \eqref{F_i} by finding a fixed point $\tilde w_i\in \Gamma_s^i$ for
\begin{equation}\label{fixpoint}
\tilde w_i=-\mathcal L_i^{-1}(r_i+R_i(\tilde w_i)).
\end{equation}
\begin{prop}
There exists $\ep_0\ll1$ and a set $R_{p,\alpha,l}$ such that for every $\ep\in (0,\ep_0)\backslash R_{p,\alpha,l}$, the equation $\mathcal F(V,w,\ep)=0$ has a solution $w(V,\ep)\in \Gamma_1$. Moreover,
\begin{equation}\label{eq6.28}
\|w(V,\ep)\|_1+\|D_Vw(V,\ep)\|_{L(\R^2,\Gamma_1)}+\|\p_\ep w(V,\ep)\|_{L(\R, \Gamma_1)}\leq Ce^{-\frac 14[\frac{c}{\ep}]}.
\end{equation}
\end{prop}
\begin{proof}
We will solve \eqref{fixpoint} for each $i$ on $\Pi_{N^i}\Gamma_s$ and show that $\sum_{p=1}^\infty\tilde w_i$ converges in $\Gamma_{s}$. Let $C_2:=\|\tilde g\|_{C^2(B_K(\R^2\times \Gamma_{s}),\R)}$, where $K>1$. By \eqref{eq3.13} and \eqref{eq6.6},
\begin{equation}\label{eq6.29}
\|\tilde g(0)\|_{s+\bar s}+\|D_V \tilde g(0)\|_{s+\bar s}\leq C'e^{-[\frac c\ep]},
\end{equation}
where $\bar s>0$ will be chosen later. We claim that for $ i\geq1$, 
\begin{eqnarray}\label{eq6.60}
&&\|\tilde w_i\|_s\leq N_i^{-\sigma}e^{-\frac 12[\frac c\ep]}\ , \ \|\tilde w_i\|_{s+\bar s}\leq N_i^{2\gamma}e^{-\frac 12[\frac c\ep]},\\\label{eq6.61}
&&\|D_V\tilde w_i\|_{L(\R^2,\Gamma_s^i)}\leq N_i^{-(\sigma-\gamma)}e^{-\frac 12[\frac c\ep]}\ ,\ \|D_V\tilde w_i\|_{L(\R^2,\Gamma_{s+\bar s}^i)}\leq N_i^{3\gamma}e^{-\frac 12[\frac c\ep]}.
\end{eqnarray}
where $\sigma>\gamma+l$. With slight abuse of notation, we write the operator norm $\|\cdot\|_{L(\R^2,\Gamma_s^i)}$ simply as $\|\cdot\|_s$ and $\Pi_{N_i}=\Pi_i$. In order to apply Lemma \ref{le6.3}, we also need estimates on $\partial_\ep w_i$, which will be postponed to the end of the proof.

For $i=1$, we look for a solution of 
\begin{equation}\label{eq6.62}
\tilde w_1=-\mathcal L^{-1}_1(r_1+R_1(\tilde w_1)).
\end{equation}
By \eqref{eq6.26}, 
\[\|\mathcal L^{-1}_1\|_{s}=\|(\Pi_{N_1}\mathcal L(0))^{-1}\|_{L(\Pi_{N_1}\Gamma_s,\Pi_{N_1}\Gamma_s)}\leq \frac{CN_1^\gamma}{\ep^{l-1}}.\]
Moreover,
\[\|r_1\|_{\Gamma_s^1}\leq \ep^2C'e^{-[\frac c\ep]}\ , \ \|R_1(\tilde w_1)\|_s\leq \ep^2C_2\|\tilde w_1\|_s^2.\]
One can easily verify for $\ep$ small enough such that
\begin{equation}\label{eq6.63}
N_1^{\sigma+\gamma}e^{-\frac 12[\frac c\ep]}\leq \ep^{-2(\sigma-\gamma)}e^{-\frac 12[\frac c\ep]}<\frac 12,\end{equation}
the right hand side of \eqref{eq6.62} defines a contraction mapping on the ball with radius $N_1^{-\sigma}e^{-\frac12[\frac c\ep]}$ in $\Gamma^1_s$. Thus, \eqref{eq6.62} has a unique solution $\tilde w_1$ such that
\begin{equation}\label{eq6.64}
\|\tilde w_1\|_{s}=\|\tilde w_1\|_{\Gamma_s^1}\leq N_1^{-\sigma}e^{-\frac12[\frac c\ep]}.
\end{equation}
To estimate $\|\tilde w_1\|_{s+\bar s}$, we use \eqref{eq6.62}, \eqref{eq6.29} and \eqref{tame2} to obtain
\begin{equation}\label{eq6.65}
\|\tilde w_1\|_{s+\bar s}\leq \frac{CN_1^\gamma}{\ep^{l-1}}\ep^2(C'e^{-[\frac c\ep]}+C(\bar s)N_1^{-\sigma}e^{-\frac 12[\frac c\ep]}\|\tilde w_1\|_{s+\bar s}),
\end{equation}
which implies
\begin{equation}\label{eq6.66}
\|\tilde w_1\|_{s+\bar s}\leq 2CC'\ep^{3-l}N_1^\gamma e^{-[\frac c\ep]}\leq N_1^{2\gamma}e^{-\frac 12[\frac c\ep]}.
\end{equation}
Let $z_1=D_V\tilde w_1$. Differentiating $\Pi_{1}\mathcal F(V,w_1,\ep)=0$ with respect to $U$ yields
\[
z_1=-(\Pi_1\mathcal L(w_1))^{-1}\Pi_1(\ep^2D_V\tilde g(V,w_1,\ep)).
\]
Since $\|w_1\|_s\leq N_1^{-\sigma}$, $\|(\Pi_{N_1}\mathcal L(w_1))^{-1}\|_s\leq\frac{2CN_1^\gamma}{\ep^{l-1}}$. Consequently,
\[\begin{aligned}
\|z_1\|_{s}\leq &2CN_1^\gamma \ep^{3-l}(\|D_V\tilde g(V,w_1,\ep)-D_V\tilde g(v,0,\ep)\|_s+\|D_V\tilde g(v,0,\ep)\|_s)\\
\leq &2CN_1^\gamma \ep^{3-l}(C_2N_1^{-\sigma}e^{-\frac 12[\frac c\ep]}+C'e^{-[\frac c\ep]})\leq N_1^{-(\sigma-\gamma)} e^{-\frac 12[\frac c\ep]},
\end{aligned}\]
and by \eqref{eq6.66} and \eqref{tame3},
\[\begin{aligned}
\|z_1\|_{s+\bar s}\leq &2CN_1^\gamma \ep^{3-l} \|D_V\tilde g(V,w_1,\ep)\|_{s+\bar s}\\
\leq & C(\bar s)N_1^\gamma \ep^{3-l}(N_1^\gamma e^{-\frac 12[\frac c\ep]}+(1+N_1^\gamma e^{-\frac 12[\frac c\ep]})N_1^{-\sigma} e^{-\frac 12[\frac c\ep]}+C'e^{-[\frac c\ep]})
\leq N_1^{2\gamma}e^{-\frac 12[\frac c\ep]}. \end{aligned}\]
Thus, we finish the proof of \eqref{eq6.60} and \eqref{eq6.61} for $i=1$.

Suppose \eqref{eq6.60} and $\eqref{eq6.61}$ hold for $i=i'$. We have
\begin{eqnarray}\label{eq6.67}
&&\| w_{i'}\|_s \leq 2N_{1}^{-\sigma}e^{-\frac 12[\frac c\ep]}\ , \ \|w_{i'}\|_{s+\bar s}\leq 2N_{i'}^{2\gamma}e^{-\frac 12[\frac c\ep]},\\\label{eq6.68}
&& \|D_V w_{i'}\|_{s}\leq 2N_{1}^{-(\sigma-\gamma)}e^{-\frac 12[\frac c\ep]}\ , \ \|D_V w_{i'}\|_{s+\bar s}\leq 2N_{i'}^{3\gamma}e^{-\frac 12[\frac c\ep]}.
\end{eqnarray}

For $i=i'+1$, by \eqref{tame1}, \eqref{tame3}, \eqref{eq6.29} and \eqref{eq6.67},
\[\begin{aligned}
\|r_{i'+1}\|_s=&\|(\Pi_{N_{i'+1}}-\Pi_{N_{i'}})\ep^2\tilde g(w_{i'})\|_s\\
\leq &\frac{\ep^2}{N_{i'}^{\bar s}}(\|\tilde g(w_{i'})-\tilde g(0)\|_{s+\bar s}+\|\tilde g(0)\|_{s+\bar s})\\
\leq &\frac{\ep^2C(\bar s)}{N_{i'+1}^{\frac{\bar s}{2}}}\big(\|D_w\tilde g\|_s\|w_{i'}\|_{s+\bar s}+\|D_w\tilde g\|_{s+\bar s}\|w_{i'}\|_s+C'e^{-[\frac c\ep]}\big)\\
\leq &  \frac{\ep^2C(\bar s)}{N_{i'+1}^{\frac{\bar s}{2}}}\big(2N_{i'}^{2\gamma}e^{-\frac 12[\frac c\ep]}+C'e^{-[\frac c\ep]}\big)\leq          \ep^2 C'(\bar s)N_{i'+1}^{\gamma-\frac{\bar s}{2}}e^{-\frac 12[\frac c\ep]}.
\end{aligned}\]
We also note $\|R_{i'+1}\|_s$ is quadratic in $\|\tilde w_{i'+1}\|_s$. Thus, we can verify 
$\mathcal L_{i'+1}^{-1}(r_{i'+1}+R_{i'+1}(\tilde w_{i'+1}))$
defines a contraction mapping on the ball with radius $N_{i'+1}^{-\sigma}e^{-\frac 12[\frac c\ep]}$ in $\Gamma_{i'+1}^s$ provided that
\begin{equation}\label{s}
\gamma+(\gamma-\frac{\bar s}{2})<-\sigma\Longleftrightarrow 4\gamma+2\sigma<\bar s. \end{equation}
By using \eqref{tame1} and \eqref{tame2}, we have
\[\begin{aligned}
\|\tilde w_{i'+1}\|_{s+\bar s}\leq &\frac{CN_{i'+1}^\gamma}{\ep^{l-1}}(\|r_{i'+1}\|_{s+\bar s}+\|R_{i'+1}(\tilde w_{i'+1})\|_{s+\bar s})\\
\leq & \frac{CN_{i'+1}^\gamma}{\ep^{l-1}}\ep^2\big[ C'e^{-[\frac c\ep]}+\|D_w\tilde g\|_sN_{i'}^{2\gamma}e^{-\frac 12[\frac c\ep]}
+\|D_w\tilde g\|_{s+\bar s}2N_1^{-\sigma}e^{-\frac 12[\frac c\ep]}\\
&+C(\bar s)\big(N_{i'}^{2\gamma}e^{-\frac 12[\frac c\ep]}(N_{i'+1}^{-\sigma}e^{-\frac 12[\frac c\ep]})^2+(N_{i'+1}^{-\sigma}e^{-\frac 12[\frac c\ep]})\|\tilde w_{i'+1}\|_{s+\bar s}\big)
\big]\\
\leq & CN_{i'+1}^\gamma\ep^{3-l}\big[C'e^{-[\frac c\ep]}+C(\bar s)N_{i'+1}^{\gamma}e^{-\frac 12[\frac c\ep]}+\frac 12\|\tilde w_{i'+1}\|_{s+\bar s}\big],
\end{aligned}\]
from which one can deduce
\begin{equation}\label{eq6.69}
\|\tilde w_{i'+1}\|_{s+\bar s}\leq N_{i'+1}^{2\gamma}e^{-\frac 12[\frac c\ep]}.
\end{equation}

Set $z_i=D_V\tilde w_i$. Differentiating $\Pi_{i'+1}\mathcal F(w_{i'}+\tilde w_{i'+1})=0$ and $\Pi_{i'}\mathcal F(w_{i'})=0$ with respect to $V$ yields
\[\Pi_{i'+1}\mathcal F_V(w_{i'+1})+\mathcal L_{i'+1}(\sum_{i=1}^{i'}z_i+z_{i'+1})=0\ , \ \Pi_{i'}\mathcal F_V(w_{i'})+\mathcal L_{i'+1}(\sum_{i=1}^{i'}z_i)=0.\]
Subtracting the second equality from the first one, we have
\[\begin{aligned}
-\mathcal L_{i'+1}z_{i'+1}=&\Pi_{i'+1}\mathcal F_V(w_{i'+1})-\Pi_{i'}\mathcal F_V(w_{i'})+\mathcal L_{i'+1}(\sum_{i=1}^{i'}z_i)-\mathcal L_{i'}(\sum_{i=1}^{i'}z_i)\\
=&(\Pi_{i'+1}-\Pi_{i'})\mathcal F_V(w_{i'+1})+(\Pi_{i'+1}-\Pi_{i'})\mathcal F_w(w_{i'+1})(\sum _{i=1}^{i'}z_i)\\
&+\Pi_{i'}(\mathcal F_V(w_{i'+1})-\mathcal F_V(w_{i'}))+\Pi_{i'}((\mathcal F_w(w_{i'+1})-\mathcal F_w(w_{i'}))(\sum_{i=1}^{i'}z_i))\\
\triangleq& I+II+III+IV.
\end{aligned}\]
By \eqref{tame1}, \eqref{tame3}, \eqref{eq6.67} and \eqref{eq6.68}, we can estimate terms on the right hand side as
\[\begin{aligned}
&\|I\|_s\leq \frac{1}{N_{i'}^{\bar s}}\|(\Pi_{i'+1}-\Pi_{i'})\mathcal F_V(w_{i'+1})\|_{s+\bar s}\leq \ep^2C(\bar s)N_{i'+1}^{2\gamma-\frac{\bar s}{2}}e^{-\frac 12[\frac c\ep]},\\
&\|II\|_s\leq \frac{\ep^2}{N_{i'}^{\bar s}}\|(\Pi_{i'+1}-\Pi_{i'})\tilde g_w(w_{i'+1})(\sum_{i=1}^{i'}z_i)\|_{s+\bar s}\leq \ep^2C(\bar s)N_{i'+1}^{2\gamma-\frac{\bar s}{2}}e^{-\frac 12[\frac c\ep]},\\
&\|III\|_s+\|IV\|_s\leq C_2\ep^2\|\tilde w_{i'+1}\|_s(1+\|\sum_{i=1}^{i'}z_i\|_s)\leq \ep^2C_2N_{i'+1}^{-\sigma}e^{-\frac 12[\frac c\ep]}.
\end{aligned}\]
Therefore, we have
\[\|z_{i'+1}\|_s\leq \frac{CN_{i'+1}^\gamma}{\ep^{l-1}}(\ep^2C(\bar s)N_{i'+1}^{2\gamma-\frac{\bar s}{2}}e^{-\frac 12[\frac c\ep]}+ \ep^2C_2N_{i'+1}^{-\sigma}e^{-\frac 12[\frac c\ep]}
)\leq N_{i'+1}^{-(\sigma-\gamma)}e^{-\frac 12[\frac c\ep]},\]
where the last inequality holds because of \eqref{s}. By \eqref{eq6.69}, \eqref{eq6.68} and \eqref{tame3}, we have
\[\|III\|_{s+\bar s}+\|IV\|_{s+\bar s}\leq \ep^2 C(\bar s)(N_{i'+1}^{2\gamma}+N_{i'+1}^{\frac{3\gamma}{2}})e^{-\frac 12[\frac c\ep]},\]
which implies
\[\|z_{i'+1}\|_{s+\bar s}\leq \ep^{3-l}CN_{i'+1}^\gamma C'(\bar s)N_{i'+1}^{2\gamma}e^{-\frac 12[\frac c\ep]}\leq N_{i'+1}^{3\gamma}e^{-\frac 12[\frac c\ep]}.\]
The proof of \eqref{eq6.60} and \eqref{eq6.61} is completed. 

Let
\[w:=\lim_{i\to \infty}w_i=\sum_{i'=1}^\infty \tilde w_{i'}\ , \ D_Vw=\lim_{i\to\infty}D_V w_i= \sum_{i'=1}^\infty z_{i'}.\]
Clearly, $w$ satisfies $\mathcal F(V,w,\ep)=0$. Moreover,
\[\|w\|_s\leq \sum_{i=1}^\infty N_i^{-\sigma}e^{-\frac 12[\frac c\ep]}\leq e^{-\frac 12[\frac c\ep]}\ , \ 
\|D_Vw\|_{s}\leq \sum_{i=1}^\infty N_i^{-(\sigma-\gamma)}e^{-\frac 12[\frac c\ep]}\leq e^{-\frac 12[\frac c\ep]}.\]
Finally, we estimate $\p_\ep \tilde w_{i}$ to fulfill the assumption in Lemma \ref{le6.3}. We claim that
\[\|\p_\ep \tilde w_i\|_s\leq \frac{N_i^{\gamma+2-\sigma}}{\ep^{l+1}}e^{-\frac 12[\frac c\ep]}\ , \ 
\|\p_\ep \tilde w_i\|_{s+\bar s}\leq \frac{N_i^{3\gamma+2}}{\ep^{l+1}}e^{-\frac 12[\frac c\ep]}.\]
The verification of the above estimates is similar to the estimate of $D_V\tilde w_i$ by differentiating $\Pi_i\frac{1}{\ep^2}\mathcal F(w_i)=0$ with respect to $\ep$. Therefore, the proof of \eqref{eq6.28} is completed.
\end{proof}

\subsection{Solution of $V$-equation}
With the solution $w(V,\ep)$ obtained in previous subsection, we solve the $V$-equation in this subsection and thus complete the proof of the Main Theorem. We begin with solving \eqref{U1}, which is
\[P_v\big(\mathcal V(P_0+\delta_1(\ep)v^\bot;w(V(\ep),\ep),\ep)-P_0\big)=0.\]
\begin{Lemma}
For any non-resonant $\ep$, there exists a function $U: (0,\ep_0)\backslash R_{p,\alpha,l}\to \R^2$ such that
\begin{equation}\label{eq6.80}
P_v\big(\mathcal V(P_0+\delta_1(\ep)v^\bot;w(V(\ep),\ep),\ep)-P_0\big)=0.\end{equation}
\end{Lemma}
\begin{proof}
Let $\mathcal T(\delta_1,\ep)\triangleq P_v\big(\mathcal V(P_0+\delta_1v^\bot;w(V(\delta_1),\ep),\ep)-P_0\big)$. Note that 
$\mathcal T(0,0)=0$ and by \eqref{eq6.28} and non-degeneracy of $P_\star$,
\[ D_{\delta_1}\mathcal T(0,0)=P_v(D_1\mathcal V(P_0;0,0)v^\bot)\ne0.\]
By \eqref{eq6.28} and the implicit function theorem, we complete the proof.
\end{proof}

We now conclude $V(\ep)$ obtained in \eqref{eq6.80} satisfies the first equation of \eqref{eq6.1} by showing 
\begin{equation}\label{eq6.82}
(I-P_v)\big(\mathcal V(P_0+\delta_1(\ep)v^\bot;w(V(\ep),\ep),\ep)-P_0\big)=\delta_1(\ep).
\end{equation}
Let $d=(I-P_v)\big(\mathcal V(P_0+\delta_1(\ep)v^\bot;w(V(\ep),\ep),\ep)-P_0\big)-\delta_1(\ep)$. Recall the definition of $H_\star$ in \eqref{eq2.13} and $v^\bot=DH_\star(P_0)$. If $d\ne0$, then
\[\big|H_\star(\mathcal V(P_0+\delta_1(\ep)v^\bot;w(V(\ep),\ep),\ep))-H_\star(P_0+\delta_1(\ep)v^\bot)\big|\geq\frac{d}{2}\|v^\bot\|^2\ne0.\]
By \eqref{eq2.7} and \eqref{eq2.13}, we have
\[\big|H(v,v_\tau,w,w_\tau,\ep)-H_\star(v,v_\tau)\big|=O(\ep^2+|w_\tau|^2+\frac{1}{\ep^2}|w|^2).\]
Since $w$ is exponentially small in $\ep$ (by \eqref{eq6.28}), we can replace \eqref{eq6.82} by
\[H\Big(\mathcal V(P_0+\delta_1(\ep)v^\bot,w(\ep),\ep),w(\ep)(p),w_\tau(\ep)(p)\Big)=H\Big(P_0+\delta_1(\ep)v^\bot,w(\ep)(0),w_\tau(\ep)(0)\Big).\]
Since $H$ is invariant under \eqref{eq6.1}, we conclude $V(\ep)$ is $p$-periodic. Finally, we recall \eqref{eq6.1} is obtained from \eqref{eq1.1} through the spatial dynamics formulation (by swapping $x$ and $t$). Together with rescaling in \eqref{eq2.2}, the proof of the Main Theorem is completed.

\section*{Acknowledgement}
The author would like to thank the anonymous referee for providing useful comments to improve the readability and presentation in this manuscript.

\end{document}